\documentclass{amsart}
\usepackage{mathtools,amsthm,verbatim}

\theoremstyle{plain}
\newtheorem{lemma}{Lemma}
\newtheorem{proposition}[lemma]{Proposition}
\newtheorem{theorem}[lemma]{Theorem}
\newtheorem{corollary}[lemma]{Corollary}
\theoremstyle{remark}
\newtheorem*{remarks}{Remarks}

\title{Each friend of 10 has at least 10 nonidentical prime factors}
\author{Henry (Maya) Robert Thackeray}
\begin{document}
\begin{abstract}
For each positive integer \(n\), if the sum of the factors of \(n\) is divided by \(n\), then the result is called the abundancy index of \(n\). If the abundancy index of some positive integer \(m\) equals the abundancy index of \(n\) but \(m\) is not equal to \(n\), then \(m\) and \(n\) are called friends. A positive integer with no friends is called solitary. The smallest positive integer that is not known to have a friend and is not known to be solitary is 10.

It is not known if the number 6 has odd friends, that is, if odd perfect numbers exist. In a 2007 article, Nielsen proved that the number of nonidentical prime factors in any odd perfect number is at least 9. A 2015 article by Nielsen, which was more complicated and used a computer program that took months to complete, increased the lower bound from 9 to 10.

This work applies methods from Nielsen's 2007 article to show that each friend of 10 has at least 10 nonidentical prime factors.

This is a formal write-up of results presented at the Southern Africa Mathematical Sciences Association Conference 2023 at the University of Pretoria.
\end{abstract}
\maketitle
Corresponding author: Henry (Maya) Robert Thackeray

Address: Department of Mathematics and Applied Mathematics, University of Pretoria, Pretoria, 0002 South Africa

Email: maya.thackeray@up.ac.za

\mbox{}

Declarations of interest: none

\mbox{}



\tableofcontents

\section{Introduction}

For each positive integer \(n\), let \(\sigma(n)\) be the sum of the factors of \(n\), and define the \emph{abundancy index} of \(n\) to be \(\sigma_{-1}(n) = \sigma(n)/n\); for example, \(\sigma_{-1}(10) = (1 + 2 + 5 + 10)/10 = 9/5\).

If \(m\) and \(n\) are positive integers, then \(m\) is called a \emph{friend} of \(n\) if \(m \neq n\) and \(\sigma_{-1}(m) = \sigma_{-1}(n)\); a positive integer \(n\) is called \emph{solitary} if no friend of \(n\) exists. This terminology appeared in a problem in the \emph{American Mathematical Monthly}, where it was conjectured that the set of solitary numbers has density zero \cite{AHG1977}. The smallest positive integer that is not known to have a friend and is not known to be solitary is 10; it seems, although it has not been proved, that for every prime \(p\) that is at least 5, the number \(2p\) is solitary (see \cite{Wnd}).

It is not known if the number 6 has odd friends, that is, if odd perfect numbers exist. In a 2007 article \cite{N2007}, Nielsen proved that the number of nonidentical prime factors in any odd perfect number is at least 9. A 2015 article by Nielsen \cite{N2015}, which was more complicated and used a computer program that took months to complete, increased the lower bound from 9 to 10.

This work applies methods from Nielsen's 2007 article \cite{N2007} to show that each friend of 10 has at least 10 nonidentical prime factors.

This is a formal write-up of results presented in a talk at the Southern Africa Mathematical Sciences Association Conference 2023, held at the University of Pretoria, Pretoria, South Africa, in November 2023.

\section{Theory}

The following definitions are used for positive integers \(m\), positive integers \(n\), and primes \(p\): the notation \(m \mid n\) means ``\(m\) divides \(n\)'', the notation \(m \hspace{-0.9mm}\not\hspace{0.9mm}\mid n\) means ``\(m\) does not divide \(n\)'', the number \(v_{p}(n)\) is the largest nonnegative integer \(v\) such that \(p^{v} \mid n\), the integer \(v_{p}(m/n)\) is \(v_{p}(m) - v_{p}(n)\), and the number \(\omega(n)\) is the number of nonidentical prime factors of \(n\).

Some well-known facts about \(\sigma_{-1}\) are recalled.

\begin{lemma}
Let \(p\) and \(q\) be primes, and let \(a\), \(b\), \(m\), and \(n\) be positive integers.
\begin{itemize}
\item[(a)] If \(a < b\), then \((p + 1)/p \leq \sigma_{-1}(p^{a}) < \sigma_{-1}(p^{b}) < p/(p - 1) =: \sigma_{-1}(p^{\infty})\).
\item[(b)] If \(p > q\), then \(\sigma_{-1}(p^{a}) < \sigma_{-1}(q^{b})\).
\item[(c)] If \(m\) and \(n\) are relatively prime, then \(\sigma_{-1}(mn) = \sigma_{-1}(m)\sigma_{-1}(n)\).
\item[(d)] If \(m \mid n\), then \(\sigma_{-1}(m) \leq \sigma_{-1}(n)\).
\item[(e)] If \(m \mid n\) and \(\sigma_{-1}(m) = \sigma_{-1}(n)\), then \(m = n\). \hfill \(\qed\)
\end{itemize}
\end{lemma}

The following foundational result underpins the later conclusions of this paper.

\begin{proposition}
Let \(n\) be a friend of \(10\).
\begin{itemize}
\item[(a)] The number \(n\) is a square, \(n\) is divisible by \(5\), and \(n\) is coprime to \(6\).
\item[(b)] The number \(n\) has at least five nonidentical prime factors.
\end{itemize}
\end{proposition}
\begin{proof}
The number \(n\) is a friend of \(10\), so \(\sigma(n)/n = 9/5\), so \(9n = 5\sigma(n)\), so \(5 \mid 9n\), so \(5 \mid n\). If \(2 \mid n\), then \(10 \mid n\) and \(\sigma_{-1}(n) = \sigma_{-1}(10)\), so \(n = 10\), which is impossible according to the definition of ``friend''; therefore, \(n\) is odd.

Let \(n = \prod_{j = 1}^{k}p_{j}^{a_{j}}\) where the numbers \(p_{j}\) are different primes, the numbers \(a_{j}\) are positive integers, and \(p_{1} < p_{2} < \cdots < p_{k}\).

Now \(9n = 5\sigma(n)\) and \(n\) is odd, so \(\sigma(n)\) is also odd. However,
\[\begin{array}{c}
\sigma(n) = \prod_{j = 1}^{k}(1 + p_{j} + p_{j}^{2} + \cdots + p_{j}^{a_{j}}) \equiv \prod_{j = 1}^{k}(a_{j} + 1) \pmod{2},
\end{array}\]
so every number \(a_{j}\) is even, so \(n\) is a square.

For a contradiction, suppose \(3 \mid n\). It follows that \(p_{1} = 3\), \(p_{2} = 5\), and \(a_{2} \geq 2\). If \(a_{2} \geq 4\), then \(\sigma_{-1}(n) \geq \sigma_{-1}(3^{2})\sigma_{-1}(5^{4}) > 9/5\), which is impossible; therefore, \(a_{2} = 2\). If \(a_{1} \geq 4\), then \(\sigma_{-1}(n) \geq \sigma_{-1}(3^{4})\sigma_{-1}(5^{2}) > 9/5\), which is impossible; therefore, \(a_{1} = a_{2} = 2\). Let \(m = n/(3^{2}5^{2})\). Now \(3^{2}\), \(5^{2}\), and \(m\) are pairwise relatively prime, so
\[\frac{\sigma(m)}{m} = \sigma_{-1}(m) = \frac{\sigma_{-1}(n)}{\sigma_{-1}(3^{2})\sigma_{-1}(5^{2})} = \frac{9/5}{(13/9)(31/25)} = \frac{405}{403},\]
so \(13 \times 31 = 403 \mid m\), so \(\sigma_{-1}(m) \geq \sigma_{-1}(403) > 405/403\), which is impossible. Therefore, \(n\) is not a multiple of \(3\).

If \(n\) has at most four nonidentical prime factors, then
\[\sigma_{-1}(n) < \sigma_{-1}(5^{\infty})\sigma_{-1}(7^{\infty})\sigma_{-1}(11^{\infty})\sigma_{-1}(13^{\infty}) = \frac{5}{4}\cdot\frac{7}{6}\cdot\frac{11}{10}\cdot\frac{13}{12} < \frac{9}{5},\]
which is impossible. Therefore, \(n\) has at least five nonidentical prime factors.
\end{proof}

\begin{remarks}
For odd perfect numbers \(n\) (that is, odd positive integers \(n\) such that \(\sigma_{-1}(n) = 2\)), similar arguments yield the following conclusions: (a) \(n\) is of the so-called Eulerian form \(p^{a}m^{2}\) where \(m\) and \(a\) are positive integers, \(p\) is prime, \(p\) does not divide \(m\), and \(p \equiv a \equiv 1 \pmod{4}\); and (b) \(n\) has at least three nonidentical prime factors (because \((3/2)(5/4) < 2 < (3/2)(5/4)(7/6)\)).

If \(p\) is a given prime such that \(p \geq 7\), then for friends \(n\) of \(2p\) (that is, positive integers \(n\) such that \(\sigma_{-1}(n) = (3/2)(p + 1)/p\)), similar arguments yield the following more trivial conclusions: (a) \(n\) is odd and divisible by \(p\), and (b) \(n\) has at least two nonidentical prime factors (because \(3/2 < (3/2)(p + 1)/p < (3/2)(5/4)\)). Imposing the additional condition \(p \equiv 1 \pmod{4}\) yields the result that \(n\) is a square, but does not improve the conclusion about \(\omega(n)\).

The reason why the arguments above easily restrict \(\omega(n)\) for friends of \(10\) but not for friends of \(2p\) where \(p \neq 5\) is: in the case of friends of \(10\), but apparently not in other cases, it can easily be proved that \(3\) does not divide \(n\). Thus, the friends-of-\(10\) problem seems suited to computer search to an extent that other similar problems are not. \hfill \(\qed\)
\end{remarks}

The following result will be applied in computer calculations; to obtain it, a well-known argument (see \cite[section 2]{CS2003} and \cite[section 6]{N2007}) is adapted to find bounds on an unknown prime factor of \(n\), where \(\sigma_{-1}(n)\) is known exactly or known to be in some given interval.

\begin{proposition}[Bounds on the smallest unknown prime]\label{propBounds}
Let \(\ell_{1}\), \(k_{1}\), and \(k\) be integers such that \(0 \leq \ell_{1} \leq k_{1} < k\). Let the positive integer \(n\) have the prime factorisation \(n = \prod_{j = 1}^{k}p_{j}^{a_{j}}\), where the numbers \(p_{j}\) are nonidentical prime numbers and the numbers \(a_{j}\) are positive integers. Suppose that \(p_{k_{1} + 1} < p_{k_{1} + 2} < \cdots < p_{k}\). Suppose that the integers \(b_{\ell_{1} + 1}\), \(b_{\ell_{1} + 2}\), \ldots, \(b_{k_{1}}\) satisfy \(a_{j} \geq b_{j}\) for \(j \in \{\ell_{1} + 1,\ell_{1} + 2, \ldots, k_{1}\}\).

(In applications, the primes \(p_{1}\) through \(p_{k_{1}}\) are known, the exponents \(a_{1}\) through \(a_{\ell_{1}}\) are known, but the other primes \(p_{j}\) and the other exponents \(a_{j}\) are not known; however, the lower bounds \(b_{j}\) are known.)

Suppose that the positive real numbers \(t_{\mathrm{min}}\) and \(t_{\mathrm{max}}\) satisfy \(t_{\mathrm{min}} \leq \sigma_{-1}(n) \leq t_{\mathrm{max}}\). Let
\[m = \frac{t_{\mathrm{min}}}{\left(\prod_{j = 1}^{\ell_{1}}\sigma_{-1}(p_{j}^{a_{j}})\right)\left(\prod_{j = \ell_{1} + 1}^{k_{1}}\sigma_{-1}(p_{j}^{\infty})\right)}\]
and
\[M = \frac{t_{\mathrm{max}}}{\left(\prod_{j = 1}^{\ell_{1}}\sigma_{-1}(p_{j}^{a_{j}})\right)\left(\prod_{j = \ell_{1} + 1}^{k_{1}}\sigma_{-1}(p_{j}^{b_{j}})\right)}.\]

It follows that
\begin{itemize}
\item[(a)] \(M > 1\) and \(1/(M - 1) \leq p_{k_{1} + 1}\);
\item[(b)] If \(a_{k_{1} + 1} \geq 2\) then
\[B_{\mathrm{low}} := \frac{1}{M - 1}\cdot\frac{8}{(2 - M)^{2} + 7} < p_{k_{1} + 1};\]
\item[(c)] If \(m > 1\) then
\[B_{\mathrm{high}} := 1 + \frac{k - k_{1}}{m - 1} > p_{k_{1} + 1};\]
and
\item[(d)] If \(k_{1} + 2 \leq k\), \(a_{k_{1} + 1} \geq 2\), and \(A\) is a real number such that \(1 < A \leq p_{k_{1} + 1}\) and \(m(A - 1)/A > 1\), then
\[g(A) := 1 + \frac{k - k_{1} - 1}{\displaystyle \frac{m(A - 1)}{A} - 1} > p_{k_{1} + 2}.\]
\end{itemize}
\end{proposition}
\begin{proof}
(a) Since \(k_{1} < k\), the hypotheses imply
\[1 < \frac{\sigma_{-1}(n)}{\left(\prod_{j = 1}^{\ell_{1}}\sigma_{-1}(p_{j}^{a_{j}})\right)\left(\prod_{j = \ell_{1} + 1}^{k_{1}}\sigma_{-1}(p_{j}^{b_{j}})\right)}\]
\[\leq \frac{t_{\mathrm{max}}}{\left(\prod_{j = 1}^{\ell_{1}}\sigma_{-1}(p_{j}^{a_{j}})\right)\left(\prod_{j = \ell_{1} + 1}^{k_{1}}\sigma_{-1}(p_{j}^{b_{j}})\right)} = M\]
and
\[\begin{array}{c}
\displaystyle M \geq \frac{M}{t_{\mathrm{max}}}\sigma_{-1}(n) \geq \frac{M}{t_{\mathrm{max}}}\left(\prod_{j = 1}^{\ell_{1}}\sigma_{-1}(p_{j}^{a_{j}})\right)\left(\prod_{j = \ell_{1} + 1}^{k_{1}}\sigma_{-1}(p_{j}^{b_{j}})\right)\sigma_{-1}(p_{k_{1} + 1}^{1})\\
\displaystyle = \sigma_{-1}(p_{k_{1} + 1}^{1}) = 1 + \frac{1}{p_{k_{1} + 1}},
\end{array}\]
so \(1/(M - 1) \leq p_{k_{1} + 1}\).

(b) Note that if a real number \(x\) satisfies \(x \geq -1\), then \(\sqrt{1 + x} \leq 1 + x/2 - x^{2}/8 + x^{3}/16\), because
\[1 + x \leq 1 + x + \left(\frac{x}{4}\right)^{4}(16 + (2 - x)^{2}) = \left(1 + \frac{x}{2} - \frac{x^{2}}{8} + \frac{x^{3}}{16}\right)^{2}\]
and
\[1 + \frac{x}{2} - \frac{x^{2}}{8} + \frac{x^{3}}{16} = \frac{5}{16} + \frac{(1 + x)(35 + (3 - 2x)^2)}{64} \geq \frac{5}{16} > 0.\]

If \(a_{k_{1} + 1} \geq 2\), then since \(k_{1} < k\), the hypotheses imply
\[\begin{array}{c}
\displaystyle M \geq \frac{M}{t_{\mathrm{max}}}\sigma_{-1}(n) \geq \frac{M}{t_{\mathrm{max}}}\left(\prod_{j = 1}^{\ell_{1}}\sigma_{-1}(p_{j}^{a_{j}})\right)\left(\prod_{j = \ell_{1} + 1}^{k_{1}}\sigma_{-1}(p_{j}^{b_{j}})\right)\sigma_{-1}(p_{k_{1} + 1}^{2})\\
\displaystyle = \sigma_{-1}(p_{k_{1} + 1}^{2}) = 1 + \frac{1}{p_{k_{1} + 1}} + \frac{1}{p_{k_{1} + 1}^{2}} > \left(1 + \frac{1}{2p_{k_{1} + 1}}\right)^{2},
\end{array}\]
so
\[1 + \frac{1}{2p_{k_{1} + 1}} < \sqrt{M} = \sqrt{1 + (M - 1)} \leq 1 + \frac{M - 1}{2} - \frac{(M - 1)^{2}}{8} + \frac{(M - 1)^{3}}{16},\]
so
\[p_{k_{1} + 1} > \frac{1/2}{(M - 1)/2 - (M - 1)^{2}/8 + (M - 1)^{3}/16} = \frac{1}{M - 1}\cdot\frac{8}{(2 - M)^{2} + 7}.\]
The result with this rational function of \(M\) is used instead of the result \(p_{k_{1} + 1} > 1/(2\sqrt{M} - 2)\) in order to maintain exact arithmetic when performing computer calculations. In practice, the previous result \(1/(M - 1) \leq p_{k_{1} + 1}\) is improved by the additional factor of \(8/((2 - M)^{2} + 7)\), since that factor is larger than \(1\) in the case where \(1 < M < 3\).

(c) If \(m > 1\), then since \(k_{1} < k\) and \(p_{k_{1} + j} \geq p_{k_{1} + 1} + j - 1\) for each \(j \in \{1,\ldots,k - k_{1}\}\), the hypotheses imply
\[m \leq \frac{m}{t_{\mathrm{min}}}\sigma_{-1}(n) < \frac{m}{t_{\mathrm{min}}}\left(\prod_{j = 1}^{\ell_{1}}\sigma_{-1}(p_{j}^{a_{j}})\right)\left(\prod_{j = \ell_{1} + 1}^{k_{1}}\sigma_{-1}(p_{j}^{\infty})\right)\prod_{j = 1}^{k - k_{1}}\left(1 + \frac{1}{p_{k_{1} + j} - 1}\right)\]
\[\leq \prod_{j = 1}^{k - k_{1}}\left(1 + \frac{1}{p_{k_{1} + 1} + j - 2}\right) = \prod_{j = 1}^{k - k_{1}}\frac{p_{k_{1} + 1} + j - 1}{p_{k_{1} + 1} + j - 2} = \frac{p_{k_{1} + 1} + k - k_{1} - 1}{p_{k_{1} + 1} - 1},\]
so \(m - 1 < (k - k_{1})/(p_{k_{1} + 1} - 1)\), from which the last required inequality follows since \(m > 1\).

(d) Use part (c) and then part (b) to obtain
\[p_{k_{1} + 2} < 1 + \frac{k - (k_{1} + 1)}{\displaystyle \frac{m}{\sigma_{-1}(p_{k_{1} + 1}^{\infty})} - 1} \leq 1 + \frac{k - k_{1} - 1}{\displaystyle \frac{m}{A/(A - 1)} - 1}.\qedhere\]
\end{proof}

For positive integers \(m\) and \(n\) such that \(\gcd(m,n) = 1\), let \(o_{n}(m)\) be the order of \(m\) modulo \(n\), that is, the smallest positive integer \(c\) such that \(m^{c} \equiv 1 \bmod n\). Each of the following two propositions is similar to a proposition in the 2007 paper of Nielsen \cite[section 3]{N2007}; to make the current article as self-contained as possible, brief proofs are given here.

\begin{proposition}\label{propValSig}
Let \(a\) be a positive integer, let \(p\) be an odd prime, and let \(x\) be an integer such that \(x > 1\) and \(p \hspace{-0.9mm}\not\hspace{0.9mm}\mid x\).
\begin{itemize}
\item[(a)] If \(p \mid x - 1\), then \(v_{p}((x^{a + 1} - 1)/(x - 1)) = v_{p}(a + 1)\).
\item[(b)] If \(p \hspace{-0.9mm}\not\hspace{0.9mm}\mid x - 1\) and \(o_{p}(x) \mid a + 1\), then
\[v_{p}((x^{a + 1} - 1)/(x - 1)) = v_{p}(x^{o_{p}(x)} - 1) + v_{p}(a + 1).\]
\item[(c)] If \(o_{p}(x) \hspace{-0.9mm}\not\hspace{0.9mm}\mid a + 1\), then \((x^{a + 1} - 1)/(x - 1)\) is a nonmultiple of \(p\). In particular, if \(a\) is even and \(p = 5 \hspace{-0.9mm}\not\hspace{0.9mm}\mid x - 1\), then \((x^{a + 1} - 1)/(x - 1)\) is a nonmultiple of \(p\).
\end{itemize}
\end{proposition}
\begin{proof}
If \(m\) is an integer such that \(m > 1\) and \(p \mid m - 1\), then \(v_{p}(m^{p} - 1) = v_{p}(m - 1) + 1\), because \(p\) is odd and by the binomial theorem,
\[m^{p} - 1 = (1 + (m - 1))^{p} - 1 = (m - 1)p + (m - 1)^{2}p\frac{p - 1}{2} + \left(\begin{array}{c}
\textrm{some multiple}\\
\textrm{of \((m - 1)^{3}\)}
\end{array}\right).\]
If \(m\) is an integer such that \(m > 1\) and \(p \mid m - 1\), and the positive integer \(r\) is a nonmultiple of \(p\), then \(v_{p}(m^{r} - 1) = v_{p}(m - 1)\), because by the binomial theorem,
\[m^{r} - 1 = (1 + (m - 1))^{r} - 1 = (m - 1)r + \textrm{(some multiple of \((m - 1)^{2}\))}.\]
The results of the last two sentences imply that if \(m\) is an integer such that \(m > 1\) and \(p \mid m - 1\), and \(r\) is a positive integer, then \(v_{p}(m^{r} - 1) = v_{p}(m - 1) + v_{p}(r)\). Each part of the result now follows from the fact that
\[v_{p}((x^{a + 1} - 1)/(x - 1)) = v_{p}(x^{a + 1} - 1) - v_{p}(x - 1).\]
(For part (a), take \(m = x\); for part (b), take \(m = x^{o_{p}(x)}\) and note that \(o_{p}(x)\), being a factor of \(p - 1\), is a nonmultiple of \(p\).)
\end{proof}

\begin{proposition}\label{propPrFactorSig}
Let \(p\) be an odd prime, and let \(a\) be an even positive integer. It follows that for each factor \(d\) of \(a + 1\) such that \(d > 1\), the number \(\sigma(p^{a})\) has a prime factor \(q_{d}\) such that \(o_{q_{d}}(p) = d\).
\end{proposition}
\begin{proof}
Note that \((x^{a + 1} - 1)/(x - 1) = \prod_{d:d > 1 \textrm{ and } d \mid a + 1}\Phi_{d}(x)\), where \(\Phi_{d}(x)\) is the \(d\)th cyclotomic polynomial. By a result of Bang \cite{B1886}, for each integer \(d > 1\) and each prime \(p\), the number \(\Phi_{d}(p)\) has a prime factor \(q\) such that the order of \(p\) modulo \(q\) is \(d\), except if \((d,p) = (1,2)\), \((d,p) = (6,2)\), or (\(d = 2\) and \(p + 1\) is a power of \(2\)); none of the three exceptions occurs if both \(d\) and \(p\) are odd.
\end{proof}

To the author's knowledge, the following two corollaries are new.

\begin{corollary}\label{corMaxExpSum}
Let \(n\) be some odd square such that \(\omega(n) = k > 1\). Let \(r\) be some prime factor of \(\gcd(n,\sigma(n))\). Let \(c\) be a nonnegative integer such that
\[\sum_{q \text{ prime}:q \mid n, \gcd(r,q(q - 1)) = 1, o_{r}(q) \mid v_{q}(n) + 1}v_{r}(q^{o_{r}(q)} - 1) \leq c.\]
Suppose that each prime factor \(q\) of \(n\) such that \(q > r\) satisfies \(v_{q}(\sigma_{-1}(n)) = 0\). It follows that \(v_{r}(n) \leq (k - 1)^{2} + c - v_{r}(\sigma_{-1}(n))\).
\end{corollary}
\begin{proof}
Let \(a = v_{r}(n)\), \(a_{-1} = v_{r}(\sigma_{-1}(n))\), and \(b = \left\lceil(a + a_{-1} - c)/(k - 1)\right\rceil\). If \(b \leq 0\), then \(a \leq c - a_{-1} \leq (k - 1)^{2} + c - a_{-1}\); from now on, suppose \(b \geq 1\).

Note that \(a + a_{-1} = v_{r}(\sigma(n))\), so by Proposition \ref{propValSig},
\[\begin{array}{rcl}
a + a_{-1} & = & v_{r}(\sigma(n/r^{a}))\\
& = & \sum_{q \text{ prime}:q \mid n, q \neq r}v_{r}(\sigma(q^{v_{q}(n)}))\\
& = & \sum_{q \text{ prime}:q \mid n, r \mid q - 1}v_{r}(\sigma(q^{v_{q}(n)}))\\
& & + \sum_{q \text{ prime}:q \mid n, \gcd(r,q(q - 1)) = 1,o_{r}(q) \mid v_{q}(n) + 1}v_{r}(\sigma(q^{v_{q}(n)}))\\
& = & \sum_{q \text{ prime}:q \mid n, q \neq r, o_{r}(q) \mid v_{q}(n) + 1}v_{r}(v_{q}(n) + 1)\\
& & + \sum_{q \text{ prime}:q \mid n, \gcd(r,q(q - 1)) = 1,o_{r}(q) \mid v_{q}(n) + 1}v_{r}(q^{o_{r}(q)} - 1)\\
& \leq & c + \sum_{q \text{ prime}:q \mid n, q \neq r, o_{r}(q) \mid v_{q}(n) + 1}v_{r}(v_{q}(n) + 1).
\end{array}\]
Let \(f\) be the number of prime factors \(q\) of \(n\) such that \(q \neq r\) and \(o_{r}(q) \mid v_{q}(n) + 1\). (This includes the prime factors \(q\) of \(n\) such that \(r \mid q - 1\).) By Proposition \ref{propValSig}, \(f \geq 1\) since \(v_{r}(\sigma(n)) \geq 1\). By the generalised pigeonhole principle, there is some prime factor \(q\) of \(n\) other than \(r\) such that \(o_{r}(q) \mid v_{q}(n) + 1\) and
\[v_{r}(v_{q}(n) + 1) \geq \frac{a + a_{-1} - c}{f} \geq \frac{a + a_{-1} - c}{k - 1},\]
so \(v_{r}(v_{q}(n) + 1) \geq b\). Now \(b \geq 1\), so the numbers \(r\), \(r^{2}\), \(r^{3}\), \ldots, \(r^{b}\) are factors of \(v_{q}(n) + 1\), so by Proposition \ref{propPrFactorSig}, \(\sigma(q^{v_{q}(n)})\) has at least \(b\) different prime factors \(q_{r}\), \(q_{r^{2}}\), \ldots, \(q_{r^{b}}\) such that for each \(i \in \{1,\ldots,b\}\), the following results hold: \(o_{q_{r^{i}}}(q) = r^{i}\), so \(r^{i} \mid q_{r^{i}} - 1\), so \(q_{r^{i}} > r\), so \(v_{q_{r^{i}}}(n) = v_{q_{r^{i}}}(\sigma(n)) > 0\). Therefore, \(b \leq k - 1\) (the \(b\) different primes \(q_{r^{i}}\) are among the \(k - 1\) different prime factors of \(n\) that are not \(r\)). It follows that
\[\frac{a + a_{-1} - c}{k - 1} \leq \left\lceil\frac{a + a_{1} - c}{k - 1}\right\rceil \leq k - 1,\]
so \(a \leq (k - 1)^{2} + c - a_{-1}\).
\end{proof}

\begin{corollary}\label{corMaxExp}
For every friend \(n\) of \(10\), if \(k = \omega(n)\) then \(v_{5}(n) \leq (k - 1)^{2} + 1\).
\end{corollary}
\begin{proof}
Apply Corollary \ref{corMaxExpSum} with \(r = 5\) and \(c = 0\); note that \(v_{5}(\sigma_{-1}(n)) = -1\) and \(v_{5}(n) \geq 2\) (since \(5 \mid n\) and \(n\) is a square), so \(v_{5}(\sigma(n)) = v_{5}(\sigma_{-1}(n)) + v_{5}(n) \geq 1\).
\end{proof}

This section ends with a proposition that will be used to speed up a computer program in two situations where, apparently, the run time would otherwise be prohibitively long. First, recall the following well-known lemma (which is equivalent to \cite[Lemma 11]{N2007}).

\begin{lemma}
Let \(p\) be an odd prime, let \(a\) be a positive integer, and let \(y\) be an integer coprime to \(p\). The equation \(z^{p - 1} \equiv 1 \bmod p^{a}\) has exactly one solution \(z\) such that \(z \equiv y \bmod p\). That solution satisfies \(z \equiv y^{p^{a - 1}} \bmod p^{a}\).
\end{lemma}
\begin{proof}
For the existence and uniqueness of \(z\), apply Hensel's lemma to \(X^{p - 1} - 1 \equiv \prod_{\widetilde{y} = 1}^{p - 1}(X - \widetilde{y}) \bmod p\). Now \((zy^{-1})^{p^{a - 1}} \equiv 1 \bmod p^{a}\) (applying Proposition \ref{propValSig}(a) if \(a > 1\)), so \(y^{p^{a - 1}} \equiv z^{p^{a - 1}} \equiv z \bmod p^{a}\).
\end{proof}

For a given odd prime \(p\) and a given positive integer \(a\), it follows that to search for integers \(x\) coprime to \(p\) such that \(p^{a - 1} < x \leq p^{a}\) and \(v_{p}(x^{p - 1} - 1) \geq a\), it is enough to check the numbers \(y^{p^{a - 1}} \bmod p^{a}\) for \(y \in \{2, \ldots, p - 1\}\). (Note that \(1^{p^{a - 1}} \equiv 1 \bmod p^{a}\).) Carrying out that check by computer for the case \(p = 31\), \(a \leq 15\) and the case \(p = 19531\), \(a \leq 7\) yields the following result, which is a slight improvement of \cite[Lemma 12]{N2007} applicable to fewer cases than that lemma.

\begin{proposition}\label{propLogCeil}
If some integer \(x\) is coprime to \(31\) and satisfies \(1 < x \leq 31^{14}\), then
\[v_{31}(x^{o_{31}(x)} - 1) \leq v_{31}(x^{30} - 1) \leq \lceil\log_{31}x\rceil + 1.\]
If some integer \(x\) is coprime to \(19531\) and satisfies \(1 < x \leq 19531^{6}\), then
\[v_{19531}(x^{o_{19531}(x)} - 1) \leq v_{19531}(x^{19530} - 1) \leq \lceil\log_{19531}x\rceil + 1.\]
\end{proposition}

\section{Computer program}

The factor-chain-search scheme used by Nielsen \cite{N2007} is adapted. The core idea is: for positive integers \(n\) and primes \(p\), if \(v_{p}(n) = a > 0\), then for each prime factor \(q\) of \(\sigma(p^{a})\), if \(v_{q}(\sigma_{-1}(n)) \leq 0\) then \(v_{q}(n) \geq v_{q}(\sigma(n)) > 0\), so \(q \mid n\). In this way, a prime factor \(p\) of \(n\), together with a known exponent \(a\), can generate other prime factors \(q\) of \(n\) under mild conditions. For example, for a friend \(n\) of \(10\), if \(v_{5}(n) = 2\) then \(v_{31}(n) = v_{31}(\sigma(n)) > 0\), and if \(v_{5}(n) = 6\) then \(v_{19531}(n) = v_{19531}(\sigma(n)) > 0\).

A SageMath computer program was implemented; the program code and output files are included as supplemental files attached to this paper.\footnote{SageMath version 10.0 (release date May 20, 2023, using Python 3.11.4) was run on Conda in Mambaforge using the Windows Subsystem for Linux on Windows 11, on a laptop with a 12th Gen Intel(R) Core(TM) i5-1235U CPU (``base speed'' 1,30 GHz with many cores; in the run with \(\omega(n) = 9\), \(v_{5}(n) = 2\), and \(v_{31}(n) \leq 94\), Windows Task Manager showed a speed of about 3,5 GHz).} A detailed description of the program follows.

The program finds candidate partial prime factorisations for all positive integers \(n\) such that
\begin{itemize}
\item \(t_{\mathrm{min}} \leq \sigma_{-1}(n) \leq t_{\mathrm{max}}\),
\item \(\omega(n) = k\), and
\item Every prime \(p\) that satisfies \(v_{p}(\sigma_{-1}(n)) > 0\) is in \(S_{\mathrm{ignore}}\),
\end{itemize}
where \(t_{\mathrm{min}}\) and \(t_{\mathrm{max}}\) are user-specified rational numbers greater than \(1\), where \(k\) is a user-specified positive integer, and where \(S_{\mathrm{ignore}}\) is a user-specified finite list of primes which this paper calls \emph{ignored primes}. The user specifies a bound \(B\), which is a cutoff value above which powers of primes are considered to be ``large''.

The program performs a depth-first search of a tree of cases. At the start of each branch of the tree:
\begin{itemize}
\item There is a known \emph{on sequence} \(S_{\mathrm{on}}\), which consists of finitely many known distinct prime factors of \(n\), which are called the \emph{on primes};
\item For each prime \(q\) in \(S_{\mathrm{on}}\), it is known that \(v_{q}(n) = a_{q}\) or it is known that \(v_{q}(n) \geq b_{q}\), where \(a_{q}\) or \(b_{q}\) respectively is a known positive integer;
\item There is a known \emph{off sequence} \(S_{\mathrm{off}}\), which consists of finitely many known distinct prime factors of \(n\), each of which is not in \(S_{\mathrm{on}}\); the primes in \(S_{\mathrm{off}}\) are called the \emph{off primes};
\item For each prime \(q\) in \(S_{\mathrm{off}}\), it is known that \(v_{q}(n) \geq b_{q}\), where \(b_{q}\) is a known positive integer; and
\item A number \(P\) is known such that for every prime factor \(q\) of \(n\), if \(q\) is neither in \(S_{\mathrm{on}}\) nor in \(S_{\mathrm{off}}\), then \(q > P\).
\end{itemize}
The values of \(S_{\mathrm{on}}\), \(S_{\mathrm{off}}\), the corresponding exponents \(a_{q}\) and \(b_{q}\), and \(P\) at the start of the program -- that is, at the root of the tree -- are specified by the user. The on and off primes (respectively, their exact exponents or minimum exponents) are some prime factors (respectively, their exact exponents or nonstrict lower bounds for their exponents) in a potential number \(n\). The (exact or minimum) exponents of the on primes have been finalised; the exponents of the off primes have not yet been finalised. Each time the program moves from one level of the tree to the next level, the number of on primes increases by exactly 1, and the number of off primes may change. The off primes are thought of as being a by-product of the on primes.

The user may choose to specify, or not to specify, a \emph{special} prime \(r\) such that \(v_{r}(\sigma_{-1}(n)) = 0\); that prime, if specified, is available for the program to use in applications of Corollary \ref{corMaxExpSum}. Specifying \(r\) is a time-saving manoeuvre: it is intended to eliminate large parts of the last two levels of the tree -- that is, the two levels furthest from the root -- for good choices of \(B\). (The author specified \(r\) in two cases that would otherwise, apparently, take a prohibitively long time to be completed by the program.) If \(r\) is specified, then 
\begin{itemize}
\item At the start of the program, \(r\) is an on prime;
\item At the start of the program, numbers \(L\) and \(\delta\) are specified by the user such that for each integer \(x\), if \(x\) is coprime to \(r\) and satisfies \(1 < x \leq L\), then \(v_{r}(x^{o_{r}(x)} - 1) \leq \lceil\log_{r}x\rceil + \delta\); and
\item Each time a new prime \(q\) appears in \(S_{\mathrm{on}} \cup S_{\mathrm{off}}\), the number
\[f_{r}(q) := \left\{\begin{array}{cl}
v_{r}(q^{o_{r}(q)} - 1) & \textrm{if \(\gcd(r,q(q - 1)) = 1\)}\\
0 & \textrm{otherwise}
\end{array}\right\}\]
is calculated and stored; for each prime \(q\) that is in \(S_{\mathrm{on}} \cup S_{\mathrm{off}}\) at the start of the program, the value of \(f_{r}(q)\) is specified by the user at the start.
\end{itemize}

In each branch of the tree, the program proceeds as follows. The function \(g\) and the numbers \(M\), \(m\), \(B_{\mathrm{low}}\), and \(B_{\mathrm{high}}\) are as in Proposition \ref{propBounds}, where the primes \(p_{1},\ldots,p_{\ell_{1}}\) are the on primes \(q\) such that it is known that \(v_{q}(n) = a_{q}\), and the primes \(p_{\ell_{1} + 1},\ldots,p_{k_{1}}\) are the other on primes and the off primes.
\begin{itemize}
\item If \(|S_{\mathrm{on}}| + |S_{\mathrm{off}}| > k\), then do not proceed further along this branch of the tree (there are too many different prime factors).
\item If \(M < 1\), then do not proceed further along this branch of the tree (\(\sigma_{-1}(n)\) is too large).
\item If \(M = 1\), then show that a solution to \(\sigma_{-1}(n) = t_{\mathrm{max}}\) has been reached and do not proceed further along this branch of the tree.
\item If \(r\) is not specified and \(|S_{\mathrm{on}}| = k\), then do the following: if \(m > 1\), then do not proceed further along this branch of the tree (\(\sigma_{-1}(n)\) is too small); otherwise, print the current data as a candidate partial prime factorisation and do not proceed further along this branch of the tree.
\item If \(r\) is specified and \(|S_{\mathrm{on}}| + |S_{\mathrm{off}}| = k\), then do the following.
\begin{itemize}
\item Calculate \(s = \sum_{q \in S_{\mathrm{on}} \cup S_{\mathrm{off}}}f_{r}(q)\).
\item If \(a_{r}\) or \(b_{r}\) is strictly larger than \((k - 1)^{2} + s\), then do not proceed further along this branch of the tree (Corollary \ref{corMaxExpSum} is violated); otherwise, print the current data as a candidate partial prime factorisation and do not proceed further along this branch of the tree.
\end{itemize}
\item If \(r\) is specified, \(|S_{\mathrm{on}}| + |S_{\mathrm{off}}| = k - 1\), and \(m > 1\), then do the following.
\begin{itemize}
\item If \(B_{\mathrm{high}} > L\), then print the current data as a candidate partial prime factorisation and do not proceed further along this branch of the tree.
\item Calculate \(s = (\sum_{q \in S_{\mathrm{on}} \cup S_{\mathrm{off}}}f_{r}(q)) + \lceil\log_{r}B_{\mathrm{high}}\rceil + \delta\).
\item If \(a_{r}\) or \(b_{r}\) is strictly larger than \((k - 1)^{2} + s\), then do not proceed further along this branch of the tree (Corollary \ref{corMaxExpSum} is violated); otherwise, print the current data as a candidate partial prime factorisation and do not proceed further along this branch of the tree.
\end{itemize}
\item If no ``do not proceed further'' instruction has been encountered in this iteration of the program and \(|S_{\mathrm{off}}| \geq 1\), then do the following.
\begin{itemize}
\item Find the smallest off prime \(p\) and its minimum exponent \(b_{p}\).
\item Let \(a\) be the smallest even positive integer such that \(a \geq b_{p}\).
\item While \(p^{a} \leq B\), do the following.
\begin{itemize}
\item Find the prime factorisation of \(\sigma(p^{a})\).
\item If every prime factor of \(\sigma(p^{a})\) is greater than \(P\) or in \(S_{\mathrm{ignore}}\) or in \(S_{\mathrm{on}}\) or in \(S_{\mathrm{off}}\), then start a new branch of the tree with the data obtained from the old branch's data as follows: move \(p\) from \(S_{\mathrm{off}}\) to \(S_{\mathrm{on}}\), let \(p\) have the exact exponent \(a_{p} = a\), and for every prime factor \(q\) of \(\sigma(p^{a})\) that is neither in \(S_{\mathrm{ignore}}\) nor in \(S_{\mathrm{on}}\), do the following: if \(q\) is in \(S_{\mathrm{off}}\), then increase the minimum exponent \(b_{q}\) of \(q\) by \(v_{q}(\sigma(p^{a}))\); otherwise, append \(q\) to the end of \(S_{\mathrm{off}}\) and let it have minimum exponent \(b_{q} = v_{q}(\sigma(p^{a}))\).
\item Increase \(a\) by \(2\).
\end{itemize}
\item Start a new branch of the tree with the data obtained from the old branch's data by moving \(p\) from \(S_{\mathrm{off}}\) to \(S_{\mathrm{on}}\) and letting \(p\) have the minimum exponent \(b_{p} = a\).
\end{itemize}
\item Otherwise, if no ``do not proceed further'' instruction has been encountered in this iteration of the program, then do the following.
\begin{itemize}
\item If \(m \leq 1\), then indicate that there is no upper bound for the next prime and do not proceed further along this branch of the tree.
\item For each prime \(p\) in the interval \(I = (\max\{P,B_{\mathrm{low}}\},B_{\mathrm{high}})\) (going through those primes \(p\) in ascending order), if \(p\) is not in \(S_{\mathrm{on}}\) then do the following. (If there are no primes in \(I\), then do not proceed further along this branch of the tree.)
\begin{itemize}
\item If \(r\) is specified, \(|S_{\mathrm{on}}| + |S_{\mathrm{off}}| = k - 2\), \(B_{\mathrm{high}} \leq L\), \(m(p - 1)/p > 1\), \(g(p) \leq L\), and \(a_{r}\) or \(b_{r}\) is strictly larger than
\[\begin{array}{c}
(k - 1)^{2} + \left(\sum_{q \in S_{\mathrm{on}} \cup S_{\mathrm{off}}}f_{r}(q)\right) + \lceil\log_{r}B_{\mathrm{high}}\rceil + \lceil\log_{r}g(p)\rceil + 2\delta,
\end{array}\]
then break out of the \(p\) for loop (Corollary \ref{corMaxExpSum} is violated for all primes \(p\) yet to be checked in this loop: \(g\) is strictly decreasing and the loop goes through the primes \(p\) in ascending order).
\item Let \(a = 2\).
\item While \(p^{a} \leq B\), do the following.
\begin{itemize}
\item Find the prime factorisation of \(\sigma(p^{a})\).
\item If every prime factor of \(\sigma(p^{a})\) is greater than \(P\) or in \(S_{\mathrm{ignore}}\) or in \(S_{\mathrm{on}}\), then start a new branch of the tree with the data obtained from the old branch's data as follows: append \(p\) to the end of \(S_{\mathrm{on}}\), let \(p\) have the exact exponent \(a_{p} = a\), let the new value of \(P\) be \(p\), and for every prime factor \(q\) of \(\sigma(p^{a})\) that is neither in \(S_{\mathrm{ignore}}\) nor in \(S_{\mathrm{on}}\), do the following: if \(q\) is in \(S_{\mathrm{off}}\), then increase the minimum exponent \(b_{q}\) of \(q\) by \(v_{q}(\sigma(p^{a}))\); otherwise, append \(q\) to the end of \(S_{\mathrm{off}}\) and let it have minimum exponent \(b_{q} = v_{q}(\sigma(p^{a}))\).
\item Increase \(a\) by \(2\).
\end{itemize}
\item Start a new branch of the tree with the data obtained from the old branch's data by appending \(p\) to the end of \(S_{\mathrm{on}}\), letting \(p\) have the minimum exponent \(b_{p} = a\), and letting the new value of \(P\) be \(p\).
\end{itemize}
\end{itemize}
\end{itemize}

\begin{table}
\begin{tabular}{lllllr}
\hline
\(k\) & \(S_{\mathrm{on}}\) & \(a_{q}\) and \(b_{q}\) for \(q\) in \(S_{\mathrm{on}}\) & \(B\) & \((r, \log_{r}L, \delta)\) & (CPU time)/s\\
\hline
5 & Empty & & \(10^{3}\) & & \(<\) 1\\
\hline
6 & Empty & & \(10^{7}\) & & \(<\) 1\\
\hline
7 & Empty & & \(10^{14}\) & & \(<\) 1\\
\hline
8 & \((5)\) & \(a_{5} = a \leq 50\) & \(*\) & & 4\\
\hline
9 & \((5, 31)\) & \(a_{5} = 2\), \(a_{31} = a \leq 94\) & \(*\) & & 7105\\
\hline
9 & \((5, 31)\) & \(a_{5} = 2\), \(b_{31} = 96\) & \(10^{16}\) & \((31, 14, 1)\) & 3\\
\hline
9 & \((5, 19531)\) & \(a_{5} = 6\), \(a_{19531} = a \leq 86\) & \(*\) & & 452\\
\hline
9 & \((5, 19531)\) & \(a_{5} = 6\), \(b_{19531} = 88\) & \(10^{17}\) & \((19531, 6, 1)\) & 61\\
\hline
9 & \((5)\) & \(a_{5} = a\): & & & \\
& & \(a = 4\) & \(10^{18}\) & & 23\\
& & \(a = 8\) & \(10^{11}\) & & \(<\) 1\\
& & \(a = 10\) & \(10^{29}\) & & 26721\\
& & \(a = 12\) & \(10^{29}\) & & 27642\\
& & \(a = 46\) & \(10^{29}\) & & 27913\\
& & \(14 \leq a \leq 64\), \(a \neq 46\) & \(*\) & & 3\\
\hline
\end{tabular}
\caption{Input-related values and CPU times in computer searches for friends \(n\) of \(10\).\label{tableResults}}
\end{table}

The SageMath program was run repeatedly, according to the specifications in Table \ref{tableResults}. Each line of the table that includes a CPU time refers to one run or to multiple runs of the SageMath program. The second and third columns of the table indicate the value of \(S_{\mathrm{on}}\) and the corresponding exponents at the start of each run.
\begin{itemize}
\item If a line of the table imposes a condition on the value of the variable \(a\), where some on prime \(p\) is listed as having exponent \(a_{p} = a\), then one run was done for each even-positive-integer value of \(a\) satisfying the given conditions, and the CPU time in that line refers to all of those runs combined. At the start of each of those runs, \(S_{\mathrm{off}}\) was taken to be the increasing sequence consisting of the different prime factors \(q\) of \(\sigma(p^{a})\) such that \(q > 5\), with the following exceptions.\footnote{If \(p = 5\), then \(\gcd(\sigma(p^{a}), 30) = 1\) since \(a\) is even, so \(S_{\mathrm{off}}\) includes every prime factor of \(\sigma(p^{a})\).}
\begin{itemize}
\item In the runs where \(k = 9\), \(S_{\mathrm{on}} = (5, 19531)\), \(a_{5} = 6\), \(a_{19531} = a\), \(30 \leq a \leq 86\), and \(a \notin \{58, 72\}\): instead of laboriously factoring \(\sigma(19531^{a})\) completely, trial division was used to express \(\sigma(19531^{a})\) as a product \(\prod_{j = 1}^{m}p_{j}^{c_{j}}\) where the primes \(p_{j}\) satisfy \(p_{1} < p_{2} < \cdots < p_{m} < 2^{31}\), or as a product \((\prod_{j = 1}^{m}p_{j}^{c_{j}})c\) where the primes \(p_{j}\) and the positive integer \(c\) satisfy \(p_{1} < p_{2} < \cdots < p_{m} < 2^{31} < c\); in both cases, the numbers \(c_{j}\) are positive integers. The sequence \(S_{\mathrm{off}}\) was taken to be the sequence obtained from \((p_{1}, \ldots, p_{m})\) by removing all terms \(p_{j}\) such that \(p_{j} \leq 5\).
\item In the two runs where \(k = 9\), \(S_{\mathrm{on}} = (5, 19531)\), \(a_{5} = 6\), and \(a_{19531} = a \in \{58, 72\}\): \(S_{\mathrm{off}}\) was taken to be the singleton sequence \((q)\), where the query \verb|19531^n-1| in the FactorDB database \cite{Fnd} provided the 15-digit smallest prime factor
\[q = 316636168836007\]
of \(\sigma(19531^{58}) = (19531^{59} - 1)/19530\) for the case \(a = 58\), as well as the 26-digit smallest prime factor
\[q = 57276919728938572349117407\]
of \(\sigma(19531^{72}) = (19531^{73} - 1)/19530\) for the case \(a = 72\).
\end{itemize}
In all cases (including the exceptions above, where trial division or FactorDB was used), at the start of the run, for each off prime \(q\), the exponent \(b_{q}\) was taken to be \(v_{q}(\sigma(p^{a}))\).
\item If a line of the table does not mention a variable \(a\), then the line refers to a single run and, at the start of that run, \(S_{\mathrm{off}}\) was taken to be empty.
\item If a line of the table has a blank space in the \((r,\log_{r}L, \delta)\) column, then the run(s) used no special prime.
\end{itemize}

The numbers \(t_{\mathrm{min}}\) and \(t_{\mathrm{max}}\) were taken to be \(9/5\). The sequence \(S_{\mathrm{ignore}}\) was taken to be the singleton sequence \((3)\). The number \(P\) was taken to be \(4\) for \(k \leq 7\) and \(5\) for \(k \geq 8\).

The values of the bound \(B\) were chosen by experimentation; that process of trial and error initially used a Magma version of the program on the online Magma calculator \cite{BCP2023}, which apparently limits computation time to 60 seconds per run. Where ``\(*\)'' appears in Table \ref{tableResults}, the values of \(B\) were as follows.
\begin{itemize}
\item \(k = 8\), \(a_{5} = a \leq 50\): For the cases \(a = 2\), \(a = 4\), \(a = 6\), or \(a = 8\), the bound \(B\) was taken to be \(10^{16}\), \(10^{11}\), \(10^{17}\), or \(10^{6}\) respectively. For other values of \(a\), the bound \(B\) was taken to be \(10^{3}\), \(10^{7}\), or \(10^{14}\) if the initial value of \(|S_{\mathrm{off}}|\) was at least \(3\), exactly \(2\), or exactly \(1\) respectively.
\item \(k = 9\), (\((a_{5}, a_{31}) = (2, a)\) or \((a_{5}, a_{19531}) = (6, a)\)): The bound \(B\) was taken to be \(10^{5}\), \(10^{11}\), or \(10^{18}\) if the initial value of \(|S_{\mathrm{off}}|\) was at least \(3\), exactly \(2\), or exactly \(1\) respectively.
\item \(k = 9\), \(a_{5} = a \in [14, 64]\), \(a \neq 46\): The bound \(B\) was taken to be \(10^{4}\), \(10^{7}\), or \(10^{14}\) if the initial value of \(|S_{\mathrm{off}}|\) was at least \(4\), exactly \(3\), or exactly \(2\) respectively.
\end{itemize}

\section{Results}

Each run of the program terminated without errors and found no candidate partial prime factorisations. The runs with \(5 \leq k \leq 7\) rule out friends \(n\) of \(10\) such that \(5 \leq \omega(n) \leq 7\). By Corollary \ref{corMaxExpSum} and Corollary \ref{corMaxExp}, the runs with \(8 \leq k \leq 9\) rule out friends \(n\) of \(10\) such that \(8 \leq \omega(n) \leq 9\). Therefore, the main theorem of this paper is proved:
\begin{theorem}
Each friend of \(10\) has at least \(10\) nonidentical prime factors.
\end{theorem}

\section{Remarks concerning run time}

Apparently, Corollary \ref{corMaxExp} is very important for the proof, because the Corollary is needed to ensure that the runs for \(k = 9\) finish in a reasonable amount of time. The runs for \(5 \leq k \leq 7\) did not use Corollary \ref{corMaxExp}, and each of those runs took less than a second. However, although a run that tackled \(k = 8\) in the same way as \(5 \leq k \leq 7\) (using \(S_{\mathrm{on}} = \emptyset\), using \(B = 10^{29}\), using \(P = 4\), using no special prime, and without using Corollary \ref{corMaxExp}) was successful, it took 28500 seconds to complete. The author believes that it is futile to try doing \(k = 9\) like this. Indeed, in two cases with \(k = 9\), Corollary \ref{corMaxExp} was not enough to reduce the run time to a manageable duration, and the results regarding special primes also needed to be used. (Many thanks to an anonymous referee for inquiring about the importance of Corollary \ref{corMaxExp}.)

The Cunningham Project tables \cite{Betal1988} confirm that for even positive integers \(a \leq 124\), the number \(\sigma(5^{a})\) has at least two different prime factors, except if \(a \in \{2, 6, 10, 12, 46\}\), in which case \(\sigma(5^{a})\) is prime. As Table \ref{tableResults} illustrates, each of the five cases
\[(\omega(n),v_{5}(n)) \in \{(9, 2), (9, 6), (9, 10), (9, 12), (9, 46)\}\]
takes several hours or uses the time-saving special prime \(r\), whereas all other runs combined take less than a minute without using the special prime \(r\). The five cases \(v_{5}(n) \in \{2, 6, 10, 12, 46\}\) would most likely be the bottlenecks for any potential attempt, using this approach, to investigate friends \(n\) of 10 such that \(10 \leq \omega(n) \leq 12\). (In order to use the 33-digit number \(\sigma(5^{46})\) as the special prime \(r\) in the case \(v_{5}(n) = 46\), a version of Proposition \ref{propLogCeil} would be needed. To prove such a statement, a direct check of all possible remainders modulo \(\sigma(5^{46})\) would not be feasible.)

As an anonymous referee kindly pointed out, the analogous run times for the three runs with \((\omega(n),v_{5}(n)) \in \{(9, 10), (9, 12), (9, 46)\}\) indicate that many calculations could be common to the three runs. This is indeed true: the large prime \(\sigma(5^{v_{5}(n)})\) contributes little to \(\sigma(n)\) in all three runs, so significant sections of the trees of cases are the same in all three runs apart from the value of that prime. (For example, consider the first few levels within the branch of the tree where \(\sigma(5^{v_{5}(n)})\) is moved from \(S_{\mathrm{off}}\) to \(S_{\mathrm{on}}\) at the first step.) However, the three runs are not completely identical (so simply exchanging the three runs for one run representing \(v_{5}(n) \geq 10\) is not possible), and they were sufficiently fast for the purposes of this paper. In general, as the referee also kindly pointed out, the program takes advantage of being able to use a specific value of \(\sigma(5^{v_{5}(n)})\) to generate tree branches.

In the runs with \((k, a_{5}) = (9, 2)\) where \(a_{31} = a \leq 94\), it appears that most of the time was used to factorise \(\sigma(31^{a})\), which is why trial division was used to avoid having to factorise \(\sigma(19531^{a})\) completely in the runs with \((k, a_{5}) = (9, 6)\) where \(a_{19531} = a \in [30, 86]\) and \(a \notin \{58, 72\}\). The numbers \(\sigma(19531^{58})\) and \(\sigma(19531^{72})\) are composite but have no prime factors less than \(2^{31}\) (the largest allowed finite upper limit in the \verb|factor_trial_division| routine in SageMath), which is why FactorDB was used to find the smallest prime factor of each of these two numbers.

\section*{Acknowledgements}

Many thanks indeed to Prof.\@ Mapundi Banda and to everyone else at the University of Pretoria for their generous continuing support.

Many thanks to the anonymous referee(s) for their much-appreciated review and comments.

This research did not receive any specific grant from funding agencies in the public, commercial, or not-for-profit sectors.

\end{document}